\documentclass[11pt,letterpaper]{article}
\usepackage[top=1in,left=1in,bottom=1in,right=1in]{geometry}

\usepackage{cite} 
\usepackage{url}  
\usepackage[colorlinks=true]{hyperref}
\usepackage{amsmath} 
\usepackage{amsthm} 
\usepackage{amssymb} 
\usepackage{algorithm}
\usepackage{algpseudocode}
\usepackage{xfrac,nicefrac}
\usepackage{graphicx}
\graphicspath{{./graphics/}}

\usepackage{tikz}
\usetikzlibrary{decorations.pathreplacing,calc,intersections,positioning,shapes.geometric}
\usepackage{forest}
\usepackage{subcaption}

\usepackage[colorinlistoftodos,textsize=tiny]{todonotes}
\usepackage{xspace}

\newtheorem{theorem}{Theorem}
\newtheorem{observ}{Observation}
\newtheorem{prop}{Proposition}
\newtheorem{corollary}{Corollary}
\newtheorem{lemma}{Lemma}
\newtheorem{remark}{Remark}
\newcommand{\lca}{\mathsf{lca}}

\newcommand{\algodescription}[3]{\noindent{\bf #1} \\%
	\begin{tabular}{@{}lp{\textwidth - 2.2cm}}%
		\emph{Input:} & #2; \\%
		\emph{Outcome:} & #3. %
	\end{tabular}%
	\vspace{0.5em}}

\DeclareMathOperator{\rt}{\rho}
\DeclareMathOperator{\mast}{\mathsf{mast}\xspace}
\DeclareMathOperator{\Seq}{\mathsf{Ord}\xspace}
\DeclareMathOperator{\Le}{\mathsf{Le}\xspace}

\DeclareMathOperator{\cP}{\mathcal{P}}

\newcommand{\bparagraph}[1]{\noindent {\textbf{#1.}}}

\newcommand{\set}[1]{\{#1\}}

\newcommand{\stone}[1]{T^{(#1)}}
\newcommand{\sttwo}[1]{S^{(#1)}}

\begin{document}

\title{On the Extremal Maximum Agreement Subtree Problem}

\newcommand{\email}{}

\author{
Alexey Markin\\
Department of Computer Science, \\
Iowa State University, USA\\
\email{amarkin@iastate.edu}
}

\date{}

\maketitle

\begin{abstract}
	Given two phylogenetic trees with the $\set{1, \ldots, n}$ leaf-set the maximum agreement subtree problem asks what is the maximum size of the subset $A \subseteq \set{1, \ldots, n}$ such that the two trees are equivalent when restricted to $A$. The long-standing extremal version of this problem focuses on the smallest number of leaves, $\mast(n)$, on which any two (binary and unrooted) phylogenetic trees with $n$ leaves must agree. In this work we prove that this number grows asymptotically as $\Theta(\log n)$; thus closing the enduring gap between the lower and upper asymptotic bounds on $\mast(n)$. 
\end{abstract}

\section{Introduction}
The algorithmic aspects of the maximum agreement subtree problem have been heavily researched for many versions of this problem (see, e.g., ~\cite{Steel:1993mast-poly,Cole:2000mast,Amir:1997mast}). The extremal problem explored in this work was first addressed more than 25 years ago by Kubicka et. al~\cite{Kubicka:1992mast}, where they proved the $c_1 (\log\log n)^{1/2}\le \mast(n) \le c_2\log n$ bounds for some constants $c_1$ and $c_2$. The lower bound was later improved to $\Omega(\log\log n)$ by Steel and Sz\'{e}kely~\cite{Steel:2009mast} and then to $\Omega(\sqrt{\log n})$ by Martin and Thatte~\cite{Martin:2013mast}. The result by Martin and Thatte originated from their proof that if at least one of the trees is either a caterpillar or a balanced tree (or an almost-balanced tree) then the maximum agreement subtree must be $\Omega(\log n)$. Additionally, Martin and Thatte conjectured that two rooted balanced trees must agree on at least $\sqrt{n}$ leaves -- this conjecture remains open.

In this work we close the gap between the lower and upper asymptotic bounds and demonstrate that $\mast(n) \in \Theta(\log n)$.
More precisely, first we prove a ``dual" (weaker) theorem stating that if any two phylogenetic trees with the $\set{1, \ldots, n}$ leaf-set are arbitrarily rooted, then they either agree as \emph{rooted} trees on $\Omega(\frac{\log n}{\log\log n})$ leaves or agree as the original unrooted trees on $\Omega(\log n)$ leaves.
Next, we extend this theorem with a more involved analysis and obtain the main result.

\section{Preliminaries}
A \emph{(phylogenetic $X$-)tree} is a binary unrooted tree with all internal nodes of degree three and leaves bijectively labeled by elements of set $X$; for convenience, we identify leaves with their labels from $X$. The set of leaves of a tree $T$ is denoted by $\Le(T)$, which is used when set $X$ is not explicitly defined. Two $X$-trees are identical if there exists a label-preserving graph isomorphism between them.
Given a set $Y \subset X$, $Y$-tree $T|Y$ is defined as the binary unrooted tree such that the minimal connected subgraph of $T$ which contains all leaves from $Y$ is a subdivision of $T|Y$. For convenience, we define the \emph{size} of a tree as $|T| := |\Le(T)| = |X|$.

A \emph{rooted (phylogenetic $X$-)tree} $T$ is a binary rooted tree with a designated root node of degree two, denoted $\rt(T)$, and each internal node having a designated left child and a right child.
Given a set $Y \subset X$ a rooted $Y$-tree $T|Y$ is defined similarly to the unrooted case. Given a node $v \in T$, $T_{v}$ denotes the subtree of $T$ rooted at $v$.

A rooted tree $T$ defines a partial order on its nodes: given two nodes $x$ and $y$ we say $x \preceq y$ if $x$ is a descendant of $y$ (and $x \prec y$ if additionally $x \ne y$). Further, we say that $x$ and $y$ are \emph{incomparable} if neither $x \preceq y$ nor $y \preceq x$. For a set $Z \subseteq X$ the \emph{least common ancestor (lca)} of $Z$, denoted $\lca_T(Z)$, is the lowest node $v$ such that each $l \in Z$ is a descendant of $v$. 

For a rooted $X$-tree $T$ let $\Seq(T)$ be the left-to-right ordering of leaves induced by the pre-order traversal of nodes of $T$. For example, $\Seq(T)$ for $T$ being the rooted tree from Figure~\ref{fig:caterpillar}~(right) is $(4,3,1,2,5)$. We refer to $\Seq(T)$ as the \emph{leaf ordering} of $T$.
We will often identify the leaves of $T$ with their indices in $\Seq(T)$.


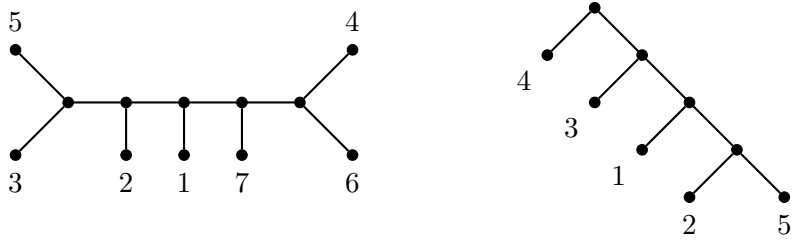
\begin{figure}
	\centering
	\begin{tikzpicture}[scale=0.7]
		\tikzstyle{vertex} = [circle,draw,fill,inner sep=0pt, minimum size=4pt]
		\tikzstyle{edge} = [draw,thick,-]
		\tikzset{vlabel/.style 2 args={#1=1pt of #2}}
		\foreach \pos/\name in {{(0,0)/v1}, {(1.1,0)/v2}, {(2.2,0)/v3},
			{(4.4,0)/v4}, {(-1,1)/l1}, {(-1,-1)/l2}, {(1.1,-1)/l3}, {(2.2,-1)/l4}, {(5.4,-1)/l5}, {(5.4,1)/l6}, {(3.3,0)/v5}, {(3.3,-1)/l7}}
			\node[vertex] (\name) at \pos {};
		\node [vlabel={above}{l1}] {5};
		\node [vlabel={below}{l2}] {3};
		\node [vlabel={below}{l3}] {2};
		\node [vlabel={below}{l4}] {1};
		\node [vlabel={below}{l5}] {6};
		\node [vlabel={above}{l6}] {4};
		\node [vlabel={below}{l7}] {7};
		\foreach \u/\v in {{v1/v2}, {v1/l1}, {v1/l2}, {v2/l3}, {v2/v3}, {v3/v4}, {v3/l4}, {v4/l5}, {v4/l6}, {v5/l7}}
			\path [edge] (\u) -- (\v);
			
		\def\d{0.9}
		\def\s{10}
		\foreach \pos/\name in {{(\s,\d+\d)/r1},{(\s-\d,\d)/rl1}, {(\s+\d,\d)/r2}, {(\s,0)/rl2}, {(\s+\d+\d,0)/r3}, {(\s+\d,-\d)/rl3}, {(\s+\d+\d+\d,-\d)/r4}, {(\s+\d+\d,-\d-\d)/rl4}, {(\s+\d+\d+\d+\d,-\d-\d)/rl5}}
			\node[vertex] (\name) at \pos {};
		\foreach \u/\v in {{r1/r4}, {r1/rl1}, {r2/rl2}, {r3/rl3}, {r4/rl4}, {r4/rl5}}
			\path [edge] (\u) -- (\v);
		\node [vlabel={below left}{rl1}] {4};
		\node [vlabel={below left}{rl2}] {3};
		\node [vlabel={below left}{rl3}] {1};
		\node [vlabel={below}{rl4}] {2};
		\node [vlabel={below}{rl5}] {5};
	\end{tikzpicture}
	\caption{Unrooted (left) and rooted (right) examples of caterpillar trees.}
	\label{fig:caterpillar}
\end{figure}

\smallskip
\bparagraph{Caterpillar trees}
An unrooted or rooted $X$-tree is a \emph{caterpillar} if every internal node (including, if present, the root) is adjacent to at least one leaf.
Figure~\ref{fig:caterpillar} demonstrates the structure of caterpillars.

\smallskip
\bparagraph{Maximum agreement subtree}
For two (unrooted or rooted) trees $T$ and $S$ on $\set{1, \ldots, n}$ leaf-set a \emph{maximum agreement set} is the maximum set $Y \subseteq \set{1, \ldots, n}$, such that $T|Y = S|Y$ up to label-preserving graph isomorphism. The tree $T|Y$ is called a \emph{maximum agreement subtree} and the size of the maximum agreement set/subtree is denoted by $\mast(T, S)$.

Let $\cP(n)$ be the set of all \emph{unrooted} $X$-trees with $X = \set{1, \ldots, n}$ then
\[
\mast(n) := \min_{T, S \in \cP(n)}(\mast(T, S)).
\]
That is, $\mast(n)$ is the minimum number of leaves on which any two unrooted $X$-trees must agree. Throughout the work we use $\log x$ to denote $\log_2 x$.

\section{Dual lower bound result}
In this section we prove a ``dual" (rooted/unrooted) lower bound result for $\mast(n)$ and lay the foundation for our main result.

Given any two unrooted $X$-trees $T$ and $S$ with $X = \set{1, \ldots, n}$ and $n \ge 4$, 
let $T'$ and $S'$ be rooted trees obtained from $T$ and $S$ respectively by rooting them at arbitrarily chosen edges $e_T \in E(T), e_S \in E(S)$ (the rooting is performed by subdividing the chosen edge with a new node and designating this node as the root). For each internal node in $T'$ and $S'$ then one of the children is designated to be the left child and the other to be the right child arbitrarily.
In this section we prove the following theorem.

\begin{theorem}
	\label{thm:mast-weaker}
	Either the rooted trees $T'$ and $S'$ have a rooted (caterpillar) agreement subtree of size at least $\frac{1}{4}\frac{\log n}{\log\log n}$ or the original unrooted trees $T$ and $S$ have a (caterpillar) agreement subtree of size at least $\log n$.
\end{theorem}

The rest of the section is dedicated to the construction proof of Theorem~\ref{thm:mast-weaker}.
To begin with, the following na\"{i}ve observation is implicitly used throughout the proof. 
\begin{observ}
	If $A$ is an agreement subtree of $T|Q$ and $S|Q$, where $Q \subset \Le(T)=\Le(S)$, then $A$ is an agreement subtree of $T$ and $S$ (in both rooted and unrooted cases).
\end{observ}

Further, Observation~\ref{obs:caterpillar} helps understanding our construction.
\begin{observ}
	\label{obs:caterpillar}
	A rooted $X$-tree is a caterpillar if and only if there exists an ordering of leaves $1, \ldots, |X|=n$ (which is unique for any caterpillar tree) such that for each $1 \le i \le n$ the least common ancestor of set $R := \set{i+1, \ldots, n}$ is strictly below the least common ancestor of set $R \cup \set{i}$ (or, equivalently, $i$ is incomparable with $\lca(R)$). Further, if such ordering is identical for two rooted $X$-trees $T'$ and $S'$ then these trees are equivalent caterpillars due to uniqueness.
\end{observ}
We now turn to the construction.

\smallskip
\bparagraph{Set up} Consider the left-to-right leaf orderings $\Seq(T')$ and $\Seq(S')$ of $T'$ and $S'$ respectively and let $\alpha$ then be a common subsequence of $\Seq(T')$ and $\Seq(S')$ (or of $\Seq(T')$ and $\Seq(S')$-reversed) of size at least $\sqrt{n}$. Note that $\alpha$ is guaranteed to exist by the Erd{\H o}s-Szekeres theorem (see~\cite{Erdos:1935}). If $\alpha$ is common to $\Seq(T')$ and $\Seq(S')$-reversed, then swap left and right children for all internal nodes in $S'$, which would then make $\alpha$ common to $\Seq(T')$ and $\Seq(S')$.
Let $X^{(1)} := \set{x \mid x \in \alpha}$ and let $\stone{1} := T'|X^{(1)}$ and $\sttwo{1} := S'|X^{(1)}$. Note that $\Seq(\stone{1}) \equiv \Seq(\sttwo{1})$.
	
For convenience of analysis, we present our construction as an iterative algorithm: on each iteration it either locates a \emph{large} ($\log n$) agreement caterpillar or adds a new leaf to an agreement set $M$ and proceeds to the next iteration with a restricted leaf-set. Next, we describe it more formally.

\smallskip
\algodescription{Iteration description.}%
{rooted $X^{(i)}$-trees $\stone{i}$ and $\sttwo{i}$ with the same leaf orderings, a set of agreement leaves $M$ with $M \cap X^{(i)} = \emptyset$}
{Either
	\begin{itemize}
		\item[(i)] finds a taxon $x \in X^{(i)}$ and a set $Y \subset X^{(i)}$, such that $\lca_{\stone{i}}(Y) \prec \lca_{\stone{i}}(Y \cup \set{x})$, $\lca_{\sttwo{i}}(Y) \prec \lca_{\sttwo{i}}(Y \cup \set{x})$, and $|Y| \ge \frac{|X^{(i)}|}{2\log n}$, or
		\item[(ii)] finds an agreement caterpillar for original trees $T$ and $S$ of size at least $\log n$.
	\end{itemize}
In the former case the construction proceeds to the next iteration by adding $x$ to $M$ and setting $\stone{i+1} = \stone{i}|Y, \sttwo{i+1} = \sttwo{i}|Y$. In the latter case the iteration stops, as an agreement subtree satisfying Theorem~\ref{thm:mast-weaker} was located}

We call a pair $(x \in X^{(i)}, Y \subset X^{(i)})$ with the properties from outcome (i) above (i.e., $\lca_{\stone{i}}(Y) \prec \lca_{\stone{i}}(Y \cup \set{x})$, $\lca_{\sttwo{i}}(Y) \prec \lca_{\sttwo{i}}(Y \cup \set{x})$, and $|Y| \ge \frac{|X^{(i)}|}{2\log n}$) a \emph{good pair}. Next, we demonstrate that such an iterative algorithm always exists.

\smallskip
\bparagraph{Construction proof}
To begin with, without loss of generality assume that the left subtree of $\stone{i}$ (subtree rooted at the left child of the root) is larger than or equal to the right subtree of $\stone{i}$ in terms of the number of nodes. If that is not the case, then swap left and right children of all internal nodes in both $\stone{i}$ and $\sttwo{i}$ -- that will preserve the equivalence of leaf orderings of $\stone{i}$ and $\sttwo{i}$. 

\begin{figure}
	\centering
	\begin{subfigure}{0.45\textwidth}
		\centering
		\begin{tikzpicture}[scale=0.7]
		\draw [] (0,0) node[below] {$Q_1$} -- (4.5,6);
		\draw[fill] (0,0) circle [radius=2pt];
		\draw [] (1.5,2) -- (3,1);
		\draw [] (3,1) -- (3.5,0) -- (2.5,0) -- (3,1);
		\node [below] at (3,0) {$Q_2$};
		\draw [] (3,4) -- (4.5,3);
		\draw [] (4.5,3) -- (5,2) -- (4,2) -- (4.5,3);
		\node [below] at (4.5,2) {$Q_3$};
		\draw [] (4.5,6) -- (6,5);
		\draw [] (6,5) -- (6.5,4) -- (5.5,4) -- (6,5);
		\node [below] at (6,4) {$Q_k$};
		\node [circle, fill, draw, inner sep=0.7pt] (d1) at (3.75,4.5) {};
		\node [circle, fill, draw, inner sep=0.7pt, above right=4pt and 3pt of d1] (d2) {};
		\node [circle, fill, draw, inner sep=0.7pt, above right=4pt and 3pt of d2] (d3) {};
		\end{tikzpicture}
	\end{subfigure}
~
	\begin{subfigure}{0.45\textwidth}
		\centering
		\begin{tikzpicture}[scale=0.7]
		\draw [] (0,0) node[below] {$R_m$} -- (-4.5,6);
		\draw[fill] (0,0) circle [radius=2pt];
		\draw [] (-1.5,2) -- (-3,1);
		\draw [] (-3,1) -- (-3.5,0) -- (-2.5,0) -- (-3,1);
		\node [below] at (-3,0) {$R_{m-1}$};
		\draw [] (-3,4) -- (-4.5,3);
		\draw [] (-4.5,3) -- (-5,2) -- (-4,2) -- (-4.5,3);
		\node [below] at (-4.5,2) {$R_{m-2}$};
		\draw [] (-4.5,6) -- (-6,5);
		\draw [] (-6,5) -- (-6.5,4) -- (-5.5,4) -- (-6,5);
		\node [below] at (-6,4) {$R_1$};
		\node [circle, fill, draw, inner sep=0.7pt] (d1) at (-3.75,4.5) {};
		\node [circle, fill, draw, inner sep=0.7pt, above left=4pt and 3pt of d1] (d2) {};
		\node [circle, fill, draw, inner sep=0.7pt, above left=4pt and 3pt of d2] (d3) {};
		\end{tikzpicture}
	\end{subfigure}	
	\caption{Schematic definition of $Q_1, \ldots, Q_k, R_1, \ldots, R_m$ subtrees from trees $\stone{i}$ (left) and $\sttwo{i}$ (right)}
	\label{fig:structure}
\end{figure}
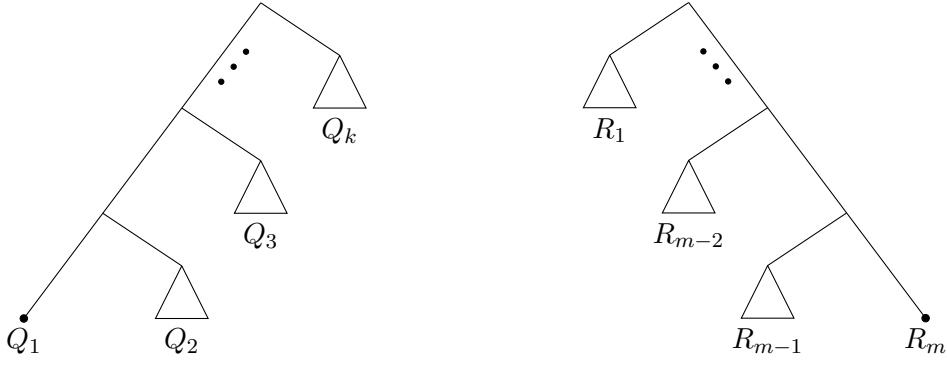

Next, let $P_1 = (u_1, \ldots, u_k)$ be the path in $\stone{i}$ from the left-most leaf to the root and $P_2 = (w_1, \ldots, w_m)$ be the path in $\sttwo{i}$ from the root to the right-most leaf. Then let $Q_1, \ldots, Q_k$ and $R_1, \ldots, R_m$ be the subtrees \emph{induced} by paths $P_1$ and $P_2$ respectively (see the illustration on Figure~\ref{fig:structure}). That is, we define $Q_j$ for some $1 \le j \le k$ (and similarly we define $R_j$) as follows:
\[
Q_j :=
\begin{cases}
	\stone{i}_{u_j} & \text{if } u_j \text{ is the lowest node in } P_1 \text { (which is } u_1 \text{ in that case)}; \\
	\stone{i}_{v} & \text{otherwise, where } v \text{ is the child of } u_j \text{ that is not on the path.}
\end{cases}
\]
Note that $Q_1$ and $R_m$ are trivial subtrees that contain only one leaf each. Additionally, note that subtrees are chosen in the way such that leaves in $Q_i$ (or $R_i$) are to the left of leaves in $Q_j$ (or $R_j$) in the common leaf ordering if $i < j$.
\begin{lemma}
	\label{lem:structure}
	For any fixed constant $C \ge 4$ at least one of the two statements always holds:
	\begin{itemize}
		\item[(i)] Exists $u \in \stone{i}, v \in \sttwo{i}$, and $x \in X^{(i)}$ such that $x \not \in \Le(\stone{i}_u)$, $x \not \in \Le(\sttwo{i}_v)$, and $\big|\Le(\stone{i}_u) \cap \Le(\sttwo{i}_v)\big| \ge \frac{|X^{(i)}|}{C}$.
		\item[(ii)] All $|Q_j|, |R_l| \le \max\big(\frac{2|X^{(i)}|}{C}, 1\big)$ for all $1 \le j \le k, 1 \le l \le m$.
	\end{itemize}
\end{lemma}
\begin{proof}
	If $\frac{2|X^{(i)}|}{C} < 2$ then the statement is trivially true (case (i) holds if $|X^{(i)}| > 1$ and case (ii) must hold otherwise). Assume $\frac{2|X^{(i)}|}{C} \ge 2$; it is sufficient to show that if (ii) does not hold, then (i) must hold. Without loss of generality assume that exists $1 < j \le k$ such that $|Q_j| > \frac{2|X^{(i)}|}{C}$ (note that $j \ne 1$, since $|Q_1| = 1$). Consider now the left and right subtrees of $\sttwo{i}$, $\sttwo{i}_l = R_1$ and $\sttwo{i}_r$ respectively; i.e., subtrees rooted at the children of $\rt(\sttwo{i})$. We consider two cases.
	\begin{itemize}
		\item $\mathbf{1 < j < k.}$ Assume that $\Le(Q_j)$ intersects with $\Le(\sttwo{i}_l)$ by at least $\frac{|X^{(i)}|}{C}$ leaves; then choose $u := \rt(Q_j)$, $v:=\rt(\sttwo{i}_l)$, and $x$ to be the \emph{right-most} leaf in the leaf ordering. Clearly, $x$ does not belong to $Q_j$, since $j < k$ and $x$ does not belong to $\sttwo{i}_l$ since $x$ is located in the right subtree of $\sttwo{i}$. That is, case (i) of our lemma holds.
		
		Otherwise, $\Le(Q_j)$ should intersect with $\Le(\sttwo{i}_r)$ by at least $\frac{|X^{(i)}|}{C}$ leaves. Then choose $u := \rt(Q_j)$, $v:=\rt(\sttwo{i}_r)$, and $x$ to be the \emph{left-most} leaf in the leaf ordering. For symmetric arguments case (i) of our lemma holds again.
		\item $\mathbf{j = k.}$ If $\Le(Q_k)$ intersects with $\Le(\sttwo{i}_r)$ by at least $\frac{|X^{(i)}|}{C}$ leaves then clearly case (i) holds if we choose $x$ to be, e.g., the left-most leaf.
		
		Otherwise, assume that it is not the case. It then follows that $|\sttwo{i}_r| < \frac{|X^{(i)}|}{C} \le \frac{|X^{(i)}|}{4}$; hence, $|\sttwo{i}_l| \ge \frac{3}{4}|X^{(i)}|$. Given our initial assumption that the left subtree of $\stone{i}$ is at least as large as its right subtree, it follows that choosing $u := \rt(\stone{i}_l)$ (root of the left subtree of $\stone{i}$), $v:= \rt(\sttwo{i}_l)$, and $x$ to be the right-most leaf satisfies conditions of case (i) of our lemma.
	\end{itemize}
\end{proof}
Note that when case (i) holds in the above lemma, the pair $(x, \Le(\stone{i}_u) \cap \Le(\sttwo{i}_v))$ is a \emph{good pair} for large enough $n$ (i.e., with $2\log n \ge C \implies n \ge 4$ when $C =4$). Additionally, note that choosing larger values of $C$ decreases the upper bound on sizes of $Q_1, \ldots, Q_k, R_1, \ldots, R_m$ subtrees, when case (i) of the lemma does not hold. We will exploit this property in the next section, when proving our main result. As for this section, we can consider $C$ to be equal $4$. 

\begin{lemma}
	\label{lem:logn-caterpillar}
	If each subtree $Q_1, \ldots, Q_k$ and $R_1, \ldots, R_m$ is of size smaller than $\frac{|X^{(i)}|}{\log n}$ then de-rooted $\stone{i}$ and $\sttwo{i}$ (and hence original $T$ and $S$) agree on a caterpillar tree of size at least $\log n$.
\end{lemma}
\begin{proof}
	We are going to construct a set $A \subseteq X^{(i)}$ such that $|A| \ge \log n$ and $A$ contains at most one leaf from each of the subtrees $Q_1, \ldots, Q_k, R_1, \ldots, R_m$. It is not then difficult to see that $\stone{i}|A$ and $\sttwo{i}|A$ (and hence $T|A$ and $S|A$) are caterpillars, which must be equivalent after de-rooting, since the leaf orderings of $\stone{i}$ and $\sttwo{i}$ are equivalent.
	
	For convenience, we identify the leaves in $\stone{i}$ and $\sttwo{i}$ with their indices, $1, \ldots, |X^{(i)}| = n^{(i)}$, from the common left-to-right leaf ordering. The leaves from each subtree $Q_1, \ldots, Q_k, R_1, \ldots, R_m$ then represent an integer interval within $[1, \ldots, n^{(i)}]$ of size at most $\frac{n^{(i)}}{\log n}$; moreover, these intervals are ordered form left to right in the same way as the subtrees are. Now construct set $A$ as follows:
	\begin{algorithm}[!h]
		\caption{$\Theta(\log n)$ agreement set}
		\begin{algorithmic}[1]
			\State $h := 1, A := \emptyset;$
			\While{$h \le n^{(i)}$}
				\State Add $h$ to $A$;
				\State Let $Q_j$ and $R_l$ be the subtrees that contain leaf $h$;
				\State Let $r_1$ and $r_2$ be the largest leaves from $Q_j$ and $R_l$ respectively;
				\State $h := \max(r_1, r_2) + 1$. \label{ln:head-jump}
			\EndWhile
		\end{algorithmic}
	\end{algorithm}

	Note that the above algorithm does not add more than one leaf to $A$ from the same subtree. Further, in Line~\ref{ln:head-jump} $h$ increases by at most $\frac{n^{(i)}}{\log n}$; thus, the size of $A$ in the end of the loop is at least $\log n$.
\end{proof}

\begin{lemma}
	\label{lem:logn-iteration}
	Assume that for some fixed $C \ge 4$ we have $|Q_j|, |R_l| \le \frac{2|X^{(i)}|}{C}$ for all $1 \le j \le k, 1 \le l \le m$ (that is, case (ii) from Lemma~\ref{lem:structure} holds) and at least one of the subtrees $Q_1, \ldots, Q_k, R_1, \ldots, R_m$ is of size at least $\frac{|X^{(i)}|}{\log n}$. Then there exists a good pair $(x, Y)$.
\end{lemma}
\begin{proof}
	The proof structure resembles the one of Lemma~\ref{lem:structure}. Without loss of generality assume that $Q_j$ is a tree of size $\ge \frac{|X^{(i)}|}{\log n}$ (note that $j \ne 1$ since $|Q_1| = 1$). We then distinguish two cases.
	
	First, assume that $j < k$; consider the left and right subtrees, $\sttwo{i}_l$ and $\sttwo{i}_r$, of $\sttwo{i}$ and let $F \in \set{\sttwo{i}_l, \sttwo{i}_r}$ be the subtree with $|\Le(Q_j) \cap \Le(F)| \ge \frac{|\Le(Q_j)|}{2}$. If $F = \sttwo{i}_l$ then choose $x$ to be the right-most leaf in the common leaf ordering. Otherwise, when $F = \sttwo{i}_r$, choose $x$ to be the left-most leaf. It is then not difficult to see that $(x, \Le(Q_j) \cap \Le(F))$ is a \emph{good pair}.

	Finally, assume that $j = k$; then note that $\sttwo{i}_l = R_1$ and by our assumption $|R_1| \le \frac{2|X^{(i)}|}{C} \le \frac{|X^{(i)}|}{2}$. Similarly, we have $|Q_k| \le \frac{|X^{(i)}|}{2}$ and given that $\Le(R_1)$ is ``on the left", while $\Le(Q_k)$ in ``on the right", we have $\Le(R_1) \cap \Le(Q_k) = \emptyset$. Then choose $x$ to be any leaf from $\Le(R_1)$ and $Y = \Le(Q_k)$. Clearly, $(x, Y)$ is a \emph{good pair}. 
\end{proof}

Combining Lemmas~\ref{lem:structure}, \ref{lem:logn-caterpillar}, and \ref{lem:logn-iteration} we have the following corollary.
\begin{corollary}
	\label{cor:weak-result}
	At least one of the following statements holds.
	\begin{itemize}
		\item[(1)] There is a good pair $(x, Y)$ with $|Y| \ge \frac{|X^{(i)}|}{C}$;
		\item[(2)] There is a ``regular" good pair $(x, Y)$ with $|Y| \ge \frac{|X^{(i)}|}{2\log n}$;
		\item[(3)] de-rooted $\stone{i}$ and $\sttwo{i}$ (and therefore original $T$ and $S$) agree on a caterpillar of size at least $\log n$.  
	\end{itemize}
\end{corollary}

While it is not necessary for this section, we distinguish cases (1) and (2) above as we will use them separately later for the proof of our main result. Corollary~\ref{cor:weak-result} then implies that we can have an algorithm fitting our original \emph{iteration description}; Algorithm~\ref{alg:weak-result} presents it.

\begin{algorithm}[!h]
	\caption{Locating an agreement caterpillar}
	\label{alg:weak-result}
	\begin{algorithmic}[1]
		\State \textbf{Input: } rooted $X^{(1)}$-trees $\stone{1}, \sttwo{1}$ with the same leaf orderings ($\Seq(\stone{1}) = \Seq(\sttwo{1})$).
		\State $M := \emptyset, i := 1;$
		\While{$|X^{(i)}| > 1$}
			\If{exists a good pair $(x, Y)$}
				\State Add $x$ to $M$ and set $X^{(i+1)} := Y, \stone{i+1} := \stone{i}|Y, \sttwo{i+1} := \sttwo{i}|Y$;
			\Else
				\State There must exist a set $A$, such that original trees $T$ and $S$ agree on $A$ and $|A| \ge \log n$; \label{ln:logn-cat-1}
				\State \Return $A$. \label{ln:logn-cat-2}
			\EndIf
		\EndWhile
		\State \Return $M$.
	\end{algorithmic}
\end{algorithm}

Note that by Observation~\ref{obs:caterpillar} $\stone{1}|M$ and $\sttwo{1}|M$ (and hence $T'|M$ and $S'|M$) must be equivalent caterpillar trees. We now find the lower bound on the size of the returned set $M$ (given that Lines~\ref{ln:logn-cat-1} and \ref{ln:logn-cat-2} are not encountered).
Note that we have $|X^{(i+1)}| \ge \frac{|X^{(i)}|}{2 \log n}$ and assume that the number iterations performed is $p$. Then $|X^{(p+1)}| = 1$ and
\begin{align*}
|X^{(p+1)}| \cdot (2 \log n)^p &= (2 \log n)^p \ge |X^{(1)}| \ge \sqrt{n}\\
p \log(2 \log n) &\ge \frac{1}{2}\log n\\
p &\ge \frac{1}{2} \frac{\log n}{\log\log n + 1} \ge \frac{1}{4}\frac{\log n}{\log\log n},
\end{align*}
with the last inequality holding for $n \ge 4$.

\begin{remark}
	\label{rem:mast-dual}
	As a corollary of Theorem~\ref{thm:mast-weaker} we have $\mast(n) \in \Omega(\frac{\log n}{\log\log n})$. However, a stronger result can be obtained as we demonstrate in the next section.
\end{remark}

\section{Asymptotics of $\mast(n)$}
We are going to refine the analysis presented in the previous section in order to obtain our main result.

\begin{theorem}
	\label{thm:mast}
	$\mast(n) \in \Theta(\log n).$
\end{theorem}

The upper bound of $\mast(n) \in O(\log(n))$ was shown by Kubicka et. al~\cite{Kubicka:1992mast}. To observe this result consider a balanced tree and a caterpillar tree; the maximum agreement subtree then must have the caterpillar shape and the size of such caterpillar is bounded by the length of the longest path in the balanced tree, which is $O(\log(n))$.

We now show that $\mast(n) \in \Omega(\log n)$. To do that we re-use the set up from the previous section. That is, we focus on rooted $X^{(1)}$-trees $\stone{1}$ and $\sttwo{1}$ with the same leaf orderings and of size at least $\sqrt{n}$. Further, we re-use a similar iteration methodology for construction of an agreement tree.

Recall that Lemma~\ref{lem:structure} from the previous section allows us to choose a constant $C$, which we set to $\mathbf{C := 40}$ in this section. We now refine Lemma~\ref{lem:logn-caterpillar}.

\begin{lemma}
	\label{lem:logn-cat-strong}
	Assume that $|X^{(i)}| \ge n^{\frac{1}{4}}$ and $|Q_j|, |R_l| \le \frac{2|X^{(i)}|}{C} = \frac{|X^{(i)}|}{20}$ for all $1 \le j \le k, 1 \le l \le m$ (i.e., case (ii) from Lemma~\ref{lem:structure} holds); then at least one of the following statements holds.
	\begin{itemize}
		\item[(i)] There exist disjoint sets $X, Y \subset X^{(i)}$ such that $|X| \ge n^{1/16}$, $|Y| \ge \frac{|X^{(i)}|}{10\log n}$, $\lca_{\stone{i}}(X)$ is incomparable with $\lca_{\stone{i}}(Y)$, and $\lca_{\sttwo{i}}(X)$ is incomparable with $\lca_{\sttwo{i}}(Y)$;
		\item[(ii)] $T$ and $S$ agree on at least $\frac{1}{48}\log n$ leaves (that induce a caterpillar).
	\end{itemize}
\end{lemma}
The proof of Lemma~\ref{lem:logn-cat-strong} uses the following result established by Martin and Thatte~\cite{Martin:2013mast} and based on the earlier work by Steel and Sz\'{e}kely~\cite{Steel:2009mast}.
\begin{prop}[Martin and Thatte~\cite{Martin:2013mast}; Steel and Sz\'{e}kely~\cite{Steel:2009mast}]
	\label{prop:caterpillar}
	Any two unrooted $X$-trees $T$ and $S$ on $n$ leaves, where $T$ is a caterpillar, have a maximum agreement subtree of size at least $\frac{1}{3}\log n$.
\end{prop}
\begin{proof}[Proof of Lemma~\ref{lem:logn-cat-strong}]
	Similar to the proof of Lemma~\ref{lem:logn-caterpillar} we identify leaves in $\stone{i}$ and $\sttwo{i}$ with their indices, $1, \ldots, |X^{(i)}| = n'$, in their common left-to-right leaf ordering. Then each subtree $Q_j$ or $R_l$ induces an integer interval inside $[1, n']$ of size at most $n'/20$. Consider now the leaves in the $I := [\frac{8}{20}n', \frac{12}{20}n']$ interval. Let $l$ be the smallest leaf (integer) from $I$ such that the subtrees $Q_{i_1}$ and $R_{h_1}$ that contain $l$ ``lie" completely within $I$ (that is, $\Le(Q_{i_1})$ and $\Le(R_{h_1})$ are within $I$). Observe that $l \le \frac{9}{20}n'$. Similarly, we define $r$ to be the largest leaf (integer) from $I$ such that subtrees $Q_{i_s}$ and $R_{h_t}$ that contain $r$ lie completely within $I$; then $r \ge \frac{11}{20}n'$. We now focus on subtrees $Q_{i_1}, \ldots, Q_{i_s}$ and $R_{h_1}, \ldots, R_{h_t}$. Observe that $\bigcup_{l = i_1}^{i_s}{\Le(Q_{l})}$ and $\bigcup_{l=h_1}^{h_t}{\Le(R_{l})}$ are supersets of $\set{l, \ldots, r}$ and $|\set{l, \ldots, r}| \ge \frac{2}{20}n'$.
	By a simple modification of Lemma~\ref{lem:logn-caterpillar} it is not difficult to see that either at least one of the $Q_{i_1}, \ldots, Q_{i_s}, R_{h_1}, \ldots, R_{h_t}$ subtrees is of size at least $\frac{(2/20)n'}{\log n}$ or $T$ and $S$ agree on a caterpillar of size at least $\log n$ (i.e., statement (ii) holds). Assume now that the former holds and let $Q_y$ be a subtree of size at least $\frac{(2/20)n'}{\log n}$. In case there is no such $Q_y$ and a subtree of size at least $\frac{(2/20)n'}{\log n}$ is among $R_{h_1}, \ldots, R_{h_t}$ subtrees, the argument that we present next changes only in one aspect, which we point out in the end.
	
	Let us now focus on the $[1, \frac{4}{20}n']$ interval. Again, by considering a simple modification of Lemma~\ref{lem:logn-caterpillar}, there must exists a subtree $Q_x$ (or $R_x$, which we disregard for symmetry) of size at least $\frac{(4/20)n'}{\log n}$ and $\Le(Q_x)$ lies within $[1, \frac{5}{20}n')$ -- otherwise, case (ii) of our Lemma holds. Recall that $n' \ge n^{\frac{1}{4}}$; then for sufficiently large $n$ we have $\frac{(4/20)n'}{\log n} \ge n^{\frac{1}{8}}$.
	
	Let now $R_{k_1}, \ldots, R_{k_f}$ be the subtrees from $\sttwo{i}$ that intersect with $Q_x$ on the leaf-set. We claim that either (1) at least one of $R_{k_l}$ subtrees intersects with $Q_x$ on at least $n^{\frac{1}{16}}$ leaves or (2) case (ii) of our lemma holds.
	Assume that all subtrees $R_{k_l}$ that intersect with $Q_x$ have $|\Le(R_{k_l}) \cap \Le(Q_x)| < n^{\frac{1}{16}}$; then, given the established $|Q_x| \ge n^{\frac{1}{8}}$ bound, the number of such subtrees must be at least $n^{\frac{1}{16}}$. Taking a single leaf from each of the $R_{k_1}, \ldots, R_{k_f}$ subtrees will then produce a set $M$ with $S|M$ being a caterpillar of size at least $n^{\frac{1}{16}}$. Proposition~\ref{prop:caterpillar} due to Martin et al.~\cite{Martin:2013mast} then implies that $T|M$ and $S|M$ (and hence $T$ and $S$) agree on a caterpillar of size at least $1/3\log(n^{\frac{1}{16}}) = 1/48\log n$; i.e., case (ii) holds.
	Otherwise, let $R_p$ be the subtree with $|\Le(Q_x) \cap \Le(R_p)| \ge n^{\frac{1}{16}}$. Clearly, $\Le(R_p)$ is within the $[1, \frac{6}{20}n')$ interval and therefore $\Le(R_p)$ does not intersect with $\Le(Q_y)$.
	
	Summing up the above arguments, define $Y := \Le(Q_y)$ and $X := \Le(Q_x) \cap \Le(R_p)$. We claim that these two sets satisfy condition (i) of our lemma. The conditions on size are satisfied by the construction; hence, we only need to confirm the incomparability conditions. Note that $\lca_{\stone{i}}(X)$ and $\lca_{\stone{i}}(Y)$ are located within the $Q_x$ and $Q_y$ subtrees respectively ($x \ne y$) and therefore are incomparable. Further, let $R_{j_1}, \ldots, R_{j_e}$ be the subtrees that intersect on leaves with the set $Y$; given that $Y$ is within $[\frac{8}{20}n', \frac{12}{20}n']$ we have $\Le(R_{j_l})$ lying within $(\frac{7}{20}n', \frac{13}{20}n')$ for all $1 \le l \le e$. Note now that $\lca_{\sttwo{i}}(X)$ is within the $R_p$ subtree and $\Le(R_p)$ precedes the $(\frac{7}{20}n', \frac{13}{20}n')$ interval. Hence, if $v_p$ is the parent of the root of $R_p$ then $\lca{\sttwo{i}}(Y)$ must be located below it and $\lca_{\sttwo{i}}(X)$ is incomparable with $\lca_{\sttwo{i}}(Y)$ (see Figure~\ref{fig:XY} for an illustration). That is, case (i) holds.
	
	Finally, we come back to the assumption that $Q_y$ exist: if it does not, then a similar subtree, $R_y$, must exist in $\sttwo{i}$ and the argument proceeds the same way with the exception that we would locate set $X$ in the $[\frac{15}{20}n', n']$ interval instead of $[1, \frac{5}{20}n']$.
\end{proof}

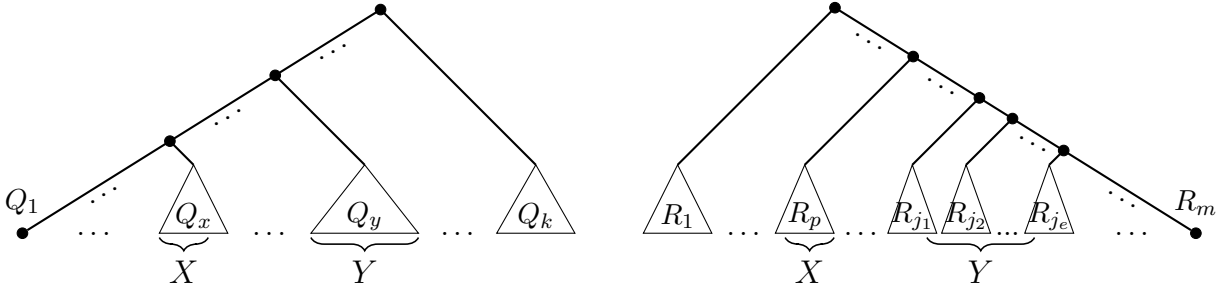
\begin{figure}
	\centering
	\begin{tikzpicture}[scale=0.65]
	\tikzstyle{vertex} = [circle,draw,fill,inner sep=0pt, minimum size=4pt]
	\tikzstyle{edge} = [draw,thick,-]
	\tikzset{vlabel/.style 2 args={#1=1pt of #2}}
	\def\step{0.8}
	\def\subtree(#1,#2,#3,#4)%
		{\draw (#1,#2) -- (#1 - #3,0) -- (#1 + #3,0) -- cycle;%
		 \node [] at (#1,10pt) {#4};}
	 
	\node [vertex] (q1) at (0,0) {};
	\coordinate [] (L1) at (8,5) {};
	\foreach \x/\r/\label/\name in {{3.5/0.7/$Q_x$/qx}, {7/1.1/$Q_y$/qy}, {10.5/0.8/$Q_k$/qk}} {
		\subtree(\x, 1.4, \r, \label);
		\coordinate [] (t) at (\x - 4,1.4 + 4) {};
		\path [name path = p1] (\x,1.4) -- (t);
		\path [name path = p2] (q1) -- (L1);
		\path [name intersections={of=p1 and p2,by=\name}];
		\path [edge] (\x,1.4) -- (\name);
		\node [vertex] at (\name) {};
	}
	\path [edge] (q1) -- (qk); 
	\node [vlabel={above}{q1}] {$Q_1$};
	\path [] (q1) -- (3.5 - 0.7,0) node [midway] {$\ldots$};
	\path [] (3.5+0.7,0) -- (7 - 1.1,0) node [midway] {$\ldots$};
	\path [] (7+1.1,0) -- (10.5 - 0.8,0) node [midway] {$\ldots$};
	\foreach \vo/\vt in {{q1/qx}, {qx/qy}, {qy/qk}}
		\path [] (\vo) -- (\vt) node [midway,rotate=30,yshift=-4pt] {$\ldots$};
	\begin{pgflowlevelscope}{\pgftransformscale{0.5}}
		\draw [decorate,decoration={brace,amplitude=10pt,mirror},yshift=-5pt,very thick] (5.6,0) -- (7.6,0) node [black,midway,yshift=-0.9cm] {\Huge$X$};
		\draw [decorate,decoration={brace,amplitude=10pt,mirror},yshift=-5pt,very thick] (31.2,0) -- (33.2,0) node [black,midway,yshift=-0.9cm] {\Huge$X$};
		\draw [decorate,decoration={brace,amplitude=10pt,mirror},yshift=-5pt,very thick] (11.8,0) -- (16.2,0) node [black,midway,yshift=-0.9cm] {\Huge$Y$};
		\draw [decorate,decoration={brace,amplitude=10pt,mirror},yshift=-5pt,very thick] (37,0) -- (41.4,0) node [black,midway,yshift=-0.9cm] {\Huge$Y$};
	\end{pgflowlevelscope}
	
	\node [vertex] (rm) at (24,0) {};
	\coordinate [] (L2) at (24-8,5) {};
	\foreach \x/\r\label/\name in {{24-3/0.5/$R_{j_e}$/rje}, {24-4.7/0.5/$R_{j_2}$/rj2}, {24-5.8/0.5/$R_{j_1}$/rj1}, {24-8/0.6/$R_{p}$/rp}, {24-10.6/0.7/$R_1$/r1}} {
		\subtree(\x, 1.4, \r, \label);
		\coordinate [] (t) at (\x + 4,1.4 + 4) {};
		\path [name path = p1] (\x,1.4) -- (t);
		\path [name path = p2] (rm) -- (L2);
		\path [name intersections={of=p1 and p2,by=\name}];
		\path [edge] (\x,1.4) -- (\name);
		\node [vertex] at (\name) {};
	}
	\path [edge] (rm) -- (r1);
	\node [vlabel={above}{rm}] {$R_m$};
	\path [] (rm) -- (24-3+0.5,0) node [midway] {$\ldots$};
	\path [] (24-3-0.5,0) -- (24-4.7+0.5,0) node [midway] {...};
	\path [] (24-5.8-0.5,0) -- (24-8+0.6,0) node [midway] {$\ldots$};
	\path [] (24-8-0.6,0) -- (24-10.6+0.7,0) node [midway] {$\ldots$};
	\foreach \vo/\vt in {{rm/rje}, {rje/rj2}, {rj1/rp}, {rp/r1}}
		\path [] (\vo) -- (\vt) node [midway,rotate=-30,yshift=-4pt] {$\ldots$};
\end{tikzpicture}
	\caption{An illustration of the potential structure of $\stone{i}$ and $\sttwo{i}$ trees for the proof of Lemma~\ref{lem:logn-cat-strong} -- e.g., when case (ii) from that lemma does not hold.}
	\label{fig:XY}
\end{figure}

If case (ii) from the above lemma holds, then Theorem~\ref{thm:mast} clearly holds as well. Otherwise, assume that case (i) holds. Due to Theorem~\ref{thm:mast-weaker} either (1) trees $\stone{i}|X$ and $\sttwo{i}|X$ agree on a rooted caterpillar of size at least
\[\frac{\log n^{\frac{1}{16}}}{4 \log\log n^{\frac{1}{16}}} = \frac{1}{16\cdot 4}\frac{\log n}{\log\log n - \log 16} \ge \frac{1}{64}\frac{\log n}{\log\log n}\]
or (2) $T$ and $S$ agree on at least a $\log n^{\frac{1}{16}} = \frac{1}{16}\log n$ caterpillar. Algorithm~\ref{alg:mast-main} summarizes all these observations. 

\begin{algorithm}[!h]
	\caption{$\Omega(\log n)$ MAST}
	\label{alg:mast-main}
	\begin{algorithmic}[1]
		\State \textbf{Input: } rooted $X^{(1)}$-trees $\stone{1}, \sttwo{1}$ with the same leaf orderings.
		\State $M := \emptyset, i := 1, C:= 40;$
		\While{$|X^{(i)}| \ge n^{\frac{1}{4}}$}
		\If{exists a good pair $(x, Y)$ with $|Y| \ge \frac{|X^{(i)}|}{C}$}
			\State Add $x$ to $M$ and set $X^{(i+1)} := Y, \stone{i+1} := \stone{i}|Y, \sttwo{i+1} := \sttwo{i}|Y$; \label{ln:iter-1}
		\ElsIf{exists a pair $(X, Y)$ as described in Lemma~\ref{lem:logn-cat-strong}, case (i)}
			\If{$\stone{i}|X$ and $\sttwo{i}|X$ agree on $M'$ leaves (rooted caterpillar)}
				\State Add leaves from $M'$ (with $|M'| \ge \frac{1}{64}\frac{\log n}{\log\log n}$) to $M$ and \label{ln:XY1}
				\State Set $X^{(i+1)} := Y, \stone{i+1} := \stone{i}|Y, \sttwo{i+1} := \sttwo{i}|Y$; \label{ln:iter-2} \label{ln:XY2}
			\Else
				\State There must exist $A$ ($|A| \ge \frac{1}{16}\log n$) such that $T$ and $S$ agree on $A$;
				\State \Return $A$. \label{ln:exit-1}
			\EndIf
		\Else
			\State There must exist a leaf-set $A$ ($|A| \ge \frac{1}{48}\log n$) such that $T$ and $S$ agree on $A$;
			\State \Return $A$. \label{ln:exit-2}
		\EndIf
		\EndWhile
		\State \Return $M$.
	\end{algorithmic}
\end{algorithm}

If the construction presented in Algorithm~\ref{alg:mast-main} exits on lines~\ref{ln:exit-1} or \ref{ln:exit-2} then we directly get at least a $\frac{1}{48}\log n$ agreement subtree. Assume now that these lines are never reached and the algorithm returns the set $M$. It is not difficult to see that $\stone{1}|M$ should be equivalent to $\sttwo{1}|M$ and therefore $T$ and $S$ agree on $M$. This can be seen by considering the following observation (a generalization of Observation~\ref{obs:caterpillar}).
\begin{observ}
	\label{obs:main-agreement}
	Let $(X_1, \ldots, X_p)$ be an ordered partition of $X$ and let rooted $X$-trees $T'$ and $S'$ have the following properties:
	\begin{itemize}
		\item $\lca_{T'}(X_i)$ is incomparable with $\displaystyle \lca_{T'}\bigg(\bigcup_{j=i+1}^{p}X_j\bigg)$ for all $1 \le i < p$;
		\item Similarly, $\lca_{S'}(X_i)$ is incomparable with $\displaystyle \lca_{S'}\bigg(\bigcup_{j=i+1}^{p}X_j\bigg)$ for all $1 \le i < p$;
		\item $T'|X_i = S'|X_i$ for all $1 \le i \le p$.
	\end{itemize}
	Then $T' = S'$.
\end{observ}

Let us now determine the lower bound on the size of $M$. Assume that line~\ref{ln:iter-1} is executed $p$ times overall, while line~\ref{ln:iter-2} is executed $q$ times. The size of $M$ is then at least $p + q \cdot \frac{1}{64}\frac{\log n}{\log \log n}$. Further, let $X^{(p+q+1)}$ be the set of leaves after the last iteration of the algorithm (i.e., $|X^{(p+q+1)}| < n^{\frac{1}{4}}$). We then have
\begin{align*}
|X^{(p+q+1)}| \cdot C^p \cdot (10 \log n)^q &\ge |X^{(1)}| \ge \sqrt{n}\\
n^{\frac{1}{4}} C^p \cdot (10 \log n)^q &\ge \sqrt{n}\\
p \log C + q(\log \log n + \log 10) &\ge \frac{1}{4}\log n\\
p &\ge \frac{1}{\log C} \big(\frac{1}{4}\log n - q\log\log n - q\log 10\big).
\end{align*}
Finally,
\[
|M| \ge p + q \cdot \frac{1}{64}\frac{\log n}{\log\log n} \ge \frac{\log n}{4\log C} + q(\frac{1}{64}\frac{\log n}{\log\log n} - \frac{\log\log n}{\log C} - \frac{\log 10}{\log C}).
\]
Note that 
\[
\frac{1}{64}\frac{\log n}{\log\log n} \ge \frac{\log\log n}{\log C} + \frac{\log 10}{\log C}
\]
for sufficiently large $n$; hence $|M| \ge \frac{\log n}{4\log C}$ for large $n$ and Theorem~\ref{thm:mast} holds.

\bibliography{complete}
\bibliographystyle{abbrv}

\end{document}